\documentclass[a4paper,11pt,titlepage,twoside]{article}

\usepackage{graphicx}
\usepackage[T1]{fontenc} 
\usepackage[utf8]{inputenc}
\usepackage[english]{babel}
\usepackage{soul}
\usepackage{amsfonts}
\usepackage{amsmath}
\usepackage{amsthm}
\usepackage{amssymb}
\usepackage{mathrsfs}
\usepackage[top=5cm, bottom=3cm, left=3.2cm, right=3.2cm]{geometry}
\usepackage{setspace}
\usepackage{afterpage}
\usepackage{extarrows}
\usepackage{fancyhdr}
\usepackage{titlesec}
\usepackage{enumitem} \setlist{nosep}
\usepackage[pdftex,breaklinks,colorlinks,linkcolor=blue,
anchorcolor=blue]{hyperref}

\theoremstyle{defin}
\newtheorem{defin}{Definition}[section]
\theoremstyle{plain}
\newtheorem{theo}[defin]{Theorem}
\newtheorem{lem}[defin]{Lemma}
\newtheorem{pro}[defin]{Proposition}
\newtheorem{cor}[defin]{Corollary}
\theoremstyle{defin}
\newtheorem{exm}[defin]{Example}
\newtheorem{rem}[defin]{Remark}

\newtheorem{quest}{Question}

\renewcommand{\O}{\Omega}
\newcommand{\D}{\mathcal{D}}
\renewcommand{\H}{\mathcal{H}}

\newcommand{\n}[1]{\|#1\|}
\newcommand{\nor}{\|\cdot\|}

\renewcommand{\l}{\langle}
\renewcommand{\r}{\rangle}
\newcommand{\N}{\mathbb{N}}
\newcommand{\Z}{\mathbb{Z}}

\newcommand{\R}{\mathbb{R}}
\newcommand{\C}{\mathbb{C}}

\newcommand{\pint}{\l\cdot,\cdot\r}
\newcommand{\pin}[2]{\l#1 , #2\r}
\newcommand{\no}{\noindent}
\newcommand{\ol}{\overline}

\newcommand{\mez}{{1/2}}

\renewcommand{\t}{\Omega}

\newcommand{\wh}{\widehat}
\newcommand{\spann}{\overline{\text{span}}}

\newcommand{\widecheck}{\rotatebox[origin=t]{-180}{$\widehat{}$}}
\newcommand{\widechec}{\rotatebox[origin=b]{-180}{$\widehat{}$}}
\newcommand{\wch}[1]{\overset{\widechec}{#1}}
\numberwithin{equation}{section}

\renewcommand\labelenumi{\emph{(\roman{enumi})}}
\renewcommand\theenumi\labelenumi

\titleformat{\section}
{\normalfont\fillast \fontsize{12}{15}\scshape}{\thesection.}{0.8em}{}

\titleformat{\subsection}
{\normalfont\fillast \fontsize{11}{12}\scshape}{\thesubsection.}{0.8em}{}

\begin{document}
	
\thispagestyle{plain}

\begin{center}
	\large
	{\uppercase{\bf Generalized frame operator, lower\\ semi-frames and sequences of translates}} \\
	\vspace*{0.5cm}
	{\scshape{Rosario Corso}}
\end{center}

\normalsize 
\vspace*{1cm}	

\small 

\begin{minipage}{12.8cm}
	{\scshape Abstract.} 
Given an arbitrary sequence of elements $\xi=\{\xi_n\}_{n\in \mathbb{N}}$ of a Hilbert space $(\H,\pint)$, the operator $T_\xi$ is defined as the operator associated to the sesquilinear form
$
\O_\xi(f,g)=\sum_{n\in \N} \pin{f}{\xi_n}\pin{\xi_n}{g},
$ 
for $f,g\in \{h\in \H: \sum_{n\in \N}|\pin{h}{\xi_n}|^2<\infty\}$. This operator is in general different from the classical frame operator but possesses some remarkable properties. For instance, $T_\xi$ is always self-adjoint in regards to a particular space, unconditionally defined and, when $\xi$ is a lower semi-frame, $T_\xi$ gives a simple expression of a dual of $\xi$. The operator $T_\xi$ and lower semi-frames are studied in the context of sequences of integer translates of a function of $L^2(\R)$. In particular, an explicit expression of $T_\xi$ is given in this context and a characterization of sequences of integer translates which are lower semi-frames is proved.
\end{minipage}

\vspace*{.5cm}

\begin{minipage}{12.8cm}
\end{minipage}

\vspace*{.5cm}

\begin{minipage}{11.8cm}
\end{minipage}

\normalsize

\section{Introduction}

A frame of a Hilbert space $\H$ with inner product $\pint$ and norm $\nor$ is a sequence of elements $\{\xi_n\}_{n\in \N}$ of $\H$ such that 
\begin{equation}
	\label{def_frame}
	A\n{f}^2\leq \sum_{n\in \N} |\pin{f}{\xi_n}|^2\leq B\n{f}^2, \qquad \forall f\in \H
\end{equation}
for some $A,B>0$ (called {\it frame bounds}). If $\{\xi_n\}_{n\in \N}$ is a frame, then the operator $S_\xi f=\sum_{n\in \N}\pin{f}{\xi_n}\xi_n$ is well-defined for every $f\in \H$, bounded, bijective and it is called the {\it frame operator} to $\{\xi_n\}_{n\in \N}$. In addition, also $\{S_\xi^{-1}\xi_n\}_{n\in \N}$ is a frame (called the {\it canonical dual frame} of $\H$) and the reconstruction formulas
\begin{equation}
	\label{rec_frame}
	f=\sum_{n\in \N}\pin{f}{\xi_n}S_\xi^{-1}\xi_n=\sum_{n\in \N}\pin{f}{S_\xi^{-1}\xi_n}\xi_n, \qquad \forall f\in \H
\end{equation}
hold.  Because of the formulas \eqref{rec_frame} frames are generalizations of orthonormal bases and they find many applications, for instance, in signal processing \cite{BHF,Vetterli}, time-frequency analysis \cite{Groechenig_b}, wavelets \cite{Daubechies}, acoustics \cite{BHNS}, quantum physics \cite{Gazeau}.

For some sequences $\{\xi_n\}_{n\in \N}$ only one inequality in \eqref{def_frame} is satisfied, but one still has reconstruction formula involving $\{\xi_n\}_{n\in \N}$. We say that $\{\xi_n\}_{n\in \N}$ is a {\it lower semi-frame} if the first inequality holds; conversely, we say that 
$\{\xi_n\}_{n\in \N}$ is a {\it Bessel sequence} if the second inequality holds. In the latter case, one can define again $S_\xi$ on the whole space and, with an abuse of terminology, $S_\xi$ is called again the frame operator of $\{\xi_n\}_{n\in \N}$. Two other standard (bounded) operators associated to a Bessel sequence $\{\xi_n\}_{n\in \N}$ are the {\it analysis operator} $C_\xi:\H\to \ell^2(\N)$ defined as $C_\xi f=\{\pin{f}{\xi_n}\}_{n\in \N}$ and the {\it synthesis operator} $D_\xi:\ell^2(\N)\to \H$ given by $D_\xi\{c_n\}_{n\in\N}=\sum_{n\in \N} c_n \xi_n$, where $\ell^2(\N)$ is the classical Hilbert space of square summable complex sequences indexed by $\N$. One has $C_\xi=D_\xi^*$ (then also $D_\xi=C_\xi^*$) and $S_\xi=D_\xi C_\xi$. 

The analysis, synthesis and frame operators can be defined for a generic sequence $\{\xi_n\}_{n\in \N}$ with the same actions but with appropriate domains \cite{Classif}. More precisely, 
$C_\xi:\D(C_\xi)\to \ell^2(\N)$ is defined on $\D(C_\xi)=\{f\in \H: \sum_{n\in \N}|\pin{f}{\xi_n}|^2<\infty\}$, 
$D_\xi:\D(D_\xi)\to \H$ is defined on $\D(D_\xi)=\{\{c_n\}\in \ell^2(\N):\sum_{n\in \N} c_n\xi_n \text{ is convergent in }\H\}$ and $S_\xi:\D(S_\xi)\to \H$ is defined on $\D(S_\xi)=\{\sum_{n\in \N} \pin{f}{\xi_n}\xi_n \text{ is convergent in }\H\}$. Moreover, $\xi$ is a Bessel sequence if and only if $\D(C_\xi)=\H$ if and only if $\D(D_\xi)=\ell^2(\N)$ if and only if $\D(S_\xi)=\H$.

Recently, in \cite{Corso_seq} a new operator $T_\xi$ associated to a sequence $\xi$ with $\D(C_\xi)$ dense have been introduced and studied. For many aspects, that we are going to discuss, it should be considered as the `correct' frame operator in the unbounded context (i.e. when $\xi$ is not Bessel).  
The operator $T_\xi$ is defined with the following argument. Assume that $\D(C_\xi)$ is dense. Then the non-negative sesquilinear form 
$$
\O_\xi(f,g)=\sum_{n\in \N} \pin{f}{\xi_n}\pin{\xi_n}{g}, \qquad f,g\in \D(C_\xi),
$$
is closed and densely defined \cite{Corso_seq}. By Kato's second representation theorem for sesquilinear forms \cite[Ch. VI, Th. 2.23]{Kato} there exists a positive self-adjoint operator $T_\xi$ such that $\D(T_\xi^\mez)=\D(C_\xi)$ (then $\D(T_\xi)\subseteq \D(C_\xi)$) and
$$
\O_\xi(f,g)=\pin{T_\xi f}{g}, \qquad f\in \D(T_\xi),g\in \D(C_\xi).
$$
In particular, $\D(T_\xi)$ is the subspace of all $f\in \D(C_\xi)$ such that the linear functional
\begin{equation}
	\label{dom_T_intro}
	g\to \sum_{n\in \N}\pin{f}{\xi_n}\pin{\xi_n}{g}, \qquad \forall  g\in \D(C_\xi)
\end{equation}
is bounded (with respect to the norm of $\H$).  Actually, 
$T_\xi=C_\xi^*C_\xi$. We say that $T_\xi$ is the {\it generalized frame operator} of $\xi$. This terminology is motivated by the fact that $T_\xi=S_\xi$ for Bessel sequences $\xi$. In general, if $\xi$ is a sequence such that $\D(C_\xi)$ is dense, then $S_\xi\subseteq T_\xi$ and the inclusion may be strict even for lower semi-frames (\cite[Examples 1, 2]{Corso_seq}). 

In \cite[Remark 1]{Corso_seq} a reconstruction formula for a lower semi-frame $\xi$ with $\D(C_\xi)$ dense was found inverting $T_\xi$ as follows 
\begin{equation}
	\label{rec_low}
	f=\sum_{n\in \N} \pin{f}{\xi_n}T_\xi^{-1}\xi_n, \qquad \forall f\in \D(C_\xi). 
\end{equation}
and $\{T_\xi^{-1}\xi_n\}_{n\in \N}$ is a Bessel sequence of $\H$ (as remarked in \cite{Corso_seq}, under the conditions above the range of $T_\xi$ is $\H$, so $T_\xi^{-1}\xi_n$ is well-defined for every $n\in \N$). 
Note the similarity between \eqref{rec_low} and the first formula in \eqref{rec_frame}. This is one of the reasons for which it is convenient to consider $T_\xi$ rather than $S_\xi$. Other motivations are that $T_\xi$ is self-adjoint, whereas $S_\xi$ could even be not closed (Example \ref{exm_S_notclosed} below). Finally, $T_\xi$ is unconditionally defined, i.e. it does not change if we consider a different ordering of the sequence, while $S_\xi$ may be conditionally convergent (Example \ref{ST_translates}). 

This paper solves some questions that arose after \cite{Corso_seq}. We present them below. 

\begin{quest}
	\label{quest1}
	Kato's representation theorem requires that the sesquilinear form is densely defined. How can we then extend the considerations above for a sequence $\xi$ such that $\D(C_\xi)$ is not dense? \\
	The question will be treated in Section \ref{subsecSeq_low_rec}. The idea is simply to restrict the problem to the closure $\ol{\D(C_\xi)}$ of $\D(C_\xi)$ and apply then Kato's theorem to obtain an operator $T_\xi:\D(T_\xi)\subseteq \ol{\D(C_\xi)}\to \ol{\D(C_\xi)}$. Formula \eqref{rec_low} becomes
	\begin{equation}
		\label{rec_low_gen}
		f=\sum_{n\in \N} \pin{f}{\xi_n}T_\xi^{-1}P\xi_n, \qquad \forall f\in \D(C_\xi). 
	\end{equation}
	where $P$ is the orthogonal projection onto $\ol{\D(C_\xi)}$. Again $\{T_\xi^{-1}P\xi_n\}_{n\in \N}$ is a Bessel sequence of $\H$. 
\end{quest}

\begin{quest}
	\label{quest2}
	A reconstruction formula for lower semi-frames $\xi$ was already established by Casazza and al. in \cite{Casazza_lower} even if $\D(C_\xi)$ is not dense. In particular, Proposition 3.4 of \cite{Casazza_lower} states that given a lower semi-frame $\xi$ of $\H$ there exists a Bessel sequence $\eta$ of $\H$ such that 
	\begin{equation}
		\label{rec_low_story}
		f=\sum_{n\in I} \pin{f}{\xi_n}\eta_n, \qquad \forall f\in \D(C_\xi).
	\end{equation}
	This formula resembles very much \eqref{rec_low_gen}, so we ask: is there any connection between \eqref{rec_low_gen} and  \eqref{rec_low_story}?\\
	We will prove in Section \ref{subsecSeq_low_rec} that $\{T_\xi^{-1}P\xi_n\}_{n\in \N}$ and the Bessel sequence $\eta$ constructed in \cite{Casazza_lower} are in fact equal.  
\end{quest}

\begin{quest}
	\label{quest3}
	From \eqref{dom_T_intro} we may ask if, when $\D(C_\xi)$ is dense, for every $f\in \D(T_\xi)$ the linear functional 
	$
	g\mapsto \sum_{n\in \N}\pin{f}{\xi_n}\pin{\xi_n}{g}
	$
	is defined and bounded in $\H$, not only in $\D(C_\xi)$. In other words, if we define the operator 
	\begin{align}
		\nonumber \D(W_\xi)&=\left \{ f\in \H:\sum_{n\in \N} \pin{f}{\xi_n} \xi_n \text{ is weakly convergent in } \H\right \}\\
		\label{def_W} &  W_\xi f=\sum_{n\in \N}\pin{f}{\xi_n}\xi_n \text{ in weak sense for }f\in \D(W_\xi),
	\end{align}
	is it true that $T_\xi=W_\xi$? \\
	The answer is negative and it will be proven with a counterexample in Section \ref{sec_TW}. In conclusion, the operator $T_\xi$ may neither coincide with $S_\xi$ nor with $W_\xi$.
\end{quest}

The rest of the paper concerns applications to the concrete {\it sequences of integer translates} $\mathcal{T}(\varphi,a):=\{\varphi_n\}_{n\in \Z}$ of a function $\varphi\in L^2(\R)$, i.e. $\varphi_n(x):=\varphi(x-na)$ for some $a>0$ and every $n\in \Z,x\in \R$. Sequences of translates are the elementary blocks of more complex sequences like Gabor systems \cite{Chris,Groechenig_b} and wavelets \cite{Daubechies}, and they are related also to the sampling theory \cite{Higgins,Vetterli}. 

First of all, we determine the generalized frame operator $T_\varphi$ of $\mathcal{T}(\varphi,a)$ in Theorem \ref{th_shift}. The expression of $T_\varphi$ is known when  $\mathcal{T}(\varphi,a)$ is a Bessel sequence and it is called {\it Walnut representation} (see for instance \cite[Theorem 10.2.1]{Chris} and \cite{Janssen}). If $\mathcal{T}(\varphi,a)$ is not a Bessel sequence, then we can again give a Walnut representation of $T_\varphi$, but of course its domain is a proper subspace of $L^2(\R)$. 

In the literature many descriptions of sequences $\mathcal{T}(\varphi,a)$ have been given in terms of the Gramian function $p_\varphi(\gamma):=\frac{1}{a}\sum_{n\in \Z}\left|\wh \varphi \left(\frac{\gamma+n}{a}\right) \right|^2$  for $\gamma\in [0,1]$. We cite some examples from \cite{Chris,heil,hsww,Ron_Shen}: $\mathcal{T}(\varphi,a)$ is a Bessel sequence if and only if $p_\varphi$ is bounded a.e. in $[0,1]$; $\mathcal{T}(\varphi,a)$ is a frame for its closed span if and only if $p_\varphi$ is bounded above and away from zero a.e. in the set $Z_\varphi:=\{\gamma \in [0,1): p_\varphi(\gamma)\neq 0\}$, i.e.\ there exist $A,B>0$ such that $A\leq p_\varphi(\gamma)\leq B$ for all $\gamma \in Z_\varphi$. Following this line, we prove in Theorem \ref{pro_lower_transl} that $\mathcal{T}(\varphi,a)$ is a lower semi-frame for its closed span if and only if $p_\varphi$ is bounded away from zero a.e. in $Z_\varphi$ and we find explicitly the dual in \eqref{rec_low} which is a sequence of translates, too. 

Some statements and examples involve {\it weighted exponentials sequences} $\mathcal{E}(g,b):=\{g_n\}_{n\in \Z}$, with $g\in L^2(0,1)$ and $b>0$, defined by $g_n(x)=g(x)e^{2 \pi i n b x}$. 
By using the Fourier transform,  $\mathcal{E}(g,b)$ can be actually considered as sequences of integer translates.

\section{Preliminaries}

Given an operator $T:\D(T)\subseteq \H_1\to \H_2$ between two Hilbert spaces we indicate by $N(T)$, $\D(T)$, $R(T)$, $\rho(T)$ and $\sigma(T)$ the kernel, the domain, the range, the resolvent set and the spectrum of $T$, respectively. When $T$ is densely defined (i.e., $D(T)$ is dense) we denote by $T^*$ its adjoint. We write $T_1\subseteq T_2$ for the operator extension. 

The symbols $\ell^2(\N),\ell^2(\Z)$ stand for the usual Hilbert spaces of complex sequences $\{c_n\}$ satisfying $\sum_{n} |c_n|^2 <\infty$. We will work also with the classical spaces $L^2(\R)$ and $L^2(0,1)$. Throughout the paper $\H$ indicates a separable Hilbert space with norm $\nor$ and inner product $\pint$. The closure and the orthogonal complement of a subspace $\mathcal{M}\subseteq \H$ are denoted by $\ol{\mathcal{M}}$ and $\mathcal{M}^\perp$, respectively. 

We will need the following notions besides those of frames. For a sequence $\xi=\{\xi_n\}_{n\in \N}$ we write span$(\{\xi_n\}_{n\in \N})$ and $\spann(\{\xi_n\}_{n\in \N})$ for its linear span and closed linear span, respectively. A sequence $\xi$ is {\it complete} if its span$(\{\xi_n\}_{n\in \N})$ is dense in $\H$, or equivalently, if the only $f\in \H$ such that $\pin{f}{\xi_n}=0$ for all $n\in \N$ is $f=0$. Note that a (lower semi-) frame is complete. 

We call $\xi$ {\it minimal} if $\xi_k\notin \spann(\{\xi_n\}_{n\neq k})$ for every $k\in \N$ or, equivalently, it admits a biorthogonal sequence in the sense of the next definition. 	Two sequences $\xi=\{\xi_n\}_{n\in \N}$ and $\eta=\{\eta_n\}_{n\in \N}$ are said to be {\it biorthogonal} if $\pin{\xi_n}{\eta_m}=\delta_{n,m}$, where $\delta_{n,m}$ is the Kronecker symbol.

A sequence $\xi$ is said to be a {\it Schauder basis} of $\H$ is for every $f\in \H$ there exists a unique complex sequence $\{a_n\}_{n\in \N}$ such that 
\begin{equation}
	\label{def_Schauder}
	f=\sum_{n=1}^\infty a_n\xi_n. 
\end{equation}
In addition, we say that $\xi$ is an {\it unconditional} Schauder basis if the series \eqref{def_Schauder} converges unconditionally (i.e. with respect to any ordering of elements) for each $f\in \H$. A Schauder basis is complete and minimal.  Moreover, the coefficients $\{a_n\}_{n\in \N}$ in \eqref{def_Schauder} are given by $\{\pin{f}{\eta_n}\}_{n\in \N}$, where $\{\eta_n\}_{n\in \N}$ is the (unique) sequence biorthogonal to $\xi$, that is also a Schauder basis of $\H$.

A {\it Riesz basis} $\xi$  is a complete sequence satisfying for some $A,B>0$
$$
A \sum_{n\in \N} |c_n|^2 \leq \left \| \sum_{n\in \N} c_n \xi_n \right \|^2 \leq B\sum_{n\in \N} |c_n|^2, \qquad\forall \{c_n\}\in \ell^2(\N).
$$
A Riesz basis is in particular a Schauder basis and a frame. 

The analysis, synthesis and frame operators $C_\xi$, $D_\xi$ and $S_\xi$ for a generic sequence have been introduced above. We write here their main properties.

\begin{pro}[{\cite[Proposition 3.3]{Classif}}]
	\label{pro_oper_CD}
	Let $\xi$ be a sequence of $\H$.  The following statements hold.	
	\begin{enumerate}
		\item $C_\xi=D_\xi^*$ and $C_\xi$ is closed.
		\item If $\D(C_\xi)$ is dense, then $D_\xi \subseteq C_\xi^*$.
		\item $D_\xi$ is closable if and only if $\D(C_\xi)$ is dense. 
		\item $D_\xi$ is closed if and only if $\D(C_\xi)$ is dense and $D_\xi=C_\xi^*$.
		\item $S_\xi=D_\xi C_\xi$.
	\end{enumerate}
\end{pro}

We will make use of the {\it pseudo-inverse} $T^\dagger$. Here $T:\D(T)\subseteq \H_1\to \H_2$ is a closed and densely defined operator between two Hilbert spaces $\H_1,\H_2$ with $R(T)$ closed. We recall (\cite[Lemma 1.1, Corollary 1.2]{Pseudo}) that $T^\dagger:\H_2\to \H_1$ is the unique bounded operator such that 
$$
N(T^\dagger)=R(T)^\perp, \quad \ol{R(T^\dagger)}=N(T)^\perp, \quad     TT^\dagger f=f, \quad \forall f\in R(T),\quad T^\dagger T g=g, \quad \forall g\in N(T)^\perp.
$$

Finally, we recall a definition related to sesquilinear form. 
Given a dense subspace $\D\subseteq \H$ and a sesquilinear form $\Omega:\D\times \D\to \C$ (i.e. a map which is linear in the first component and anti-linear in the second one), the operator $T$ with 
\begin{equation}
	\label{op_ass}
	\D(T)=\{f\in \D: \exists h\in \H, \O(f,g)=\pin{h}{g} \text{ for all } g\in \D\}
\end{equation}
and $Tf=h$, where $h$ is as in \eqref{op_ass}, is called the {\it operator associated} to $\Omega$. The density of $\D$ ensures that $T$ is well-defined and it is the greatest operator satisfying the representation 
$$
\O(f,g)=\pin{T f}{g}, \qquad \forall f\in \D(T), g\in \D.
$$

\section{The generalized frame operator}
\label{secSeq_1seq}

\begin{defin}
	Let $\xi=\{\xi_n\}_{n\in \N}$ be a sequence of $\H$. The {\it sesquilinear form associated to $\xi$} is 
	$$
	\t_\xi(f,g):=\sum_{n\in \N} \pin{f}{\xi_n}\pin{\xi_n}{g}, \qquad \forall f,g\in\D(C_\xi).
	$$
\end{defin}
\no Clearly, $\t_\xi$ is nonnegative and it is defined on the largest possible domain of the type $\D\times \D$ with $\D\subseteq \H$. Moreover
\begin{equation}
	\label{O_xi_max}
	\t_\xi(f,g)=\pin{C_\xi f}{C_\xi g}_2, \qquad\forall f,g\in \D(C_\xi).
\end{equation}
It follows by the Cauchy-Schwarz inequality that $\t_\xi$ is  unconditionally convergent, i.e. does not depend neither on the ordering of the sequence, nor the particular choice of the index set $\N$ (in other words, $\Omega_\xi$ does not change if a different index set is used for the elements of the sequence).

Since $C_\xi$ is a closed operator, $\t_\xi$ is also a closed nonnegative form. 
We can think to $\t_\xi$ as a form in the Hilbert space $\H_\xi:=\ol{\D(C_\xi)}[\nor]$ supported in $\ol{\D(C_\xi)}$ with the topology induced by $\H$. 
Under this assumption, $\t_\xi$ is densely defined and then we can talk about the operator $T_\xi$ associated to $\t_\xi$. In particular, $T_\xi:\D(T_\xi)\subseteq \H_\xi\to \H_\xi$ is an operator on $\H_\xi$ defined as follows: 
$\D(T_\xi)$ is the subspace consisting of all $f\in \D(C_\xi)$ such that the linear functional
\begin{equation*}
	g\to \sum_{n\in \N}\pin{f}{\xi_n}\pin{\xi_n}{g}, \qquad g\in \D(C_\xi)
\end{equation*}
is bounded (with respect to the norm of $\H$)
and for $f\in \D(T_\xi)$ one has $T_\xi f=h$ where $h$ is the unique element of $\H_\xi$ satisfying $\sum_{n\in \N} \pin{f}{\xi_n}\pin{\xi_n}{g}=\pin{h}{g}$ for all $g\in \D(C_\xi)$. Since $\t_\xi$ is unconditionally convergent, the operator $T_\xi$ neither depends on the particular ordering of the sequence $\xi$ nor the choice of the index set $\N$. We say then that $T_\xi$ is {\it unconditionally defined}.

By Kato's second representation theorem \cite[Ch. VI, Th. 2.23]{Kato}, $T_\xi$ is positive and self-adjoint in the space $\H_\xi$,  we have also that $\D(T_\xi^\mez)=\D(C_\xi)$ and
\begin{align*}
	\t_\xi(f,g)=\pin{T_\xi^\mez f}{T_\xi^\mez g}, \qquad \forall f, g\in \D(C_\xi),
\end{align*}
where $T_\xi^\mez:\D(T_\xi^\mez)\subseteq \H_\xi\to \H_\xi$ is the positive square-root of $T_\xi$. 
If we think $C_\xi$ as operator $C_\xi:\D(C_\xi)\subseteq \H_\xi \to \ell^2(\N)$, then we can consider its adjoint, denoted by $C_\xi^\times$, $C^\times_\xi:\D(C^\times_\xi)\subseteq \ell^2(\N)\to \H_\xi$, which is densely defined since $C_\xi$ is closed\footnote{We use the symbol $C^\times_\xi$ to denote the adjoint of $C_\xi$ as operator $C_\xi:\D(C_\xi)\subseteq \H_\xi \to \ell^2(\N)$; if $C_\xi$ is densely defined, then $\H_\xi=\H$ and we prefer to write $C^*_\xi$ instead of $C^\times_\xi$.}. From  (\ref{O_xi_max}) one can easily see that \begin{equation}
	\label{T=C^*C}
	T_\xi={C_\xi^\times}C_\xi=|C_\xi|^2, 
\end{equation} the square of the modulus of $C_\xi$. 

If $\xi$ is a Bessel sequence, then  $T_\xi=S_\xi$, of course.  
If $R(S_\xi)\subseteq \H_\xi$ (this is certainly true if $\D(C_\xi)$ is dense), then $S_\xi\subseteq T_\xi$.  Indeed for any $f\in \D(S_\xi)$ we have that 
$\sum_{n\in \N} \pin{f}{\xi_n}\pin{\xi_n}{g} =\pin{S_\xi f}{g}$ is bounded for $g\in \D(C_\xi)$. However, in \cite[Example 1]{Corso_seq} it was proved that $S_\xi$ may be strictly smaller than $T_\xi$. For this reason we give the following definition:
\begin{defin}
	Let $\xi=\{\xi_n\}_{n\in \N}$ be a sequence of $\H$. The operator $T_\xi$ constructed as above is called the {\it generalized frame operator} of $\xi$. 
\end{defin}

Throughout the rest of the paper we will give a special attention to lower semi-frames, starting with the following characterization.

\begin{pro}[{\cite[Proposition 3]{Corso_seq}}]
	\label{car_seq_T_xi}
	
	Let $\xi$ be a sequence of $\H$ and $A>0$. The following statements are equivalent.
	\begin{enumerate}
		\item 
		$\xi$ is a lower semi-frame of $\H$ with lower bound $A$;
		
		\item 
		$T_\xi$ is bounded from below by $A$, i.e., 
		$$
		\|T_\xi f\|\geq A\|f\|,\;\;  \forall f\in \D(T_\xi);
		$$
		
		\item 
		$T_\xi$ is invertible with range $\H_\xi$, $T_\xi^{-1}:\H_\xi\to \D(T_\xi)$ is bounded with respect to the norm of $\H_\xi$ and $\|T_\xi^{-1}\|\leq A$.
	\end{enumerate}
\end{pro}

\section{Lower semi-frames and reconstruction formulas}
\label{subsecSeq_low_rec}

In this section we follow the idea in \cite[Remark 1]{Corso_seq} to obtain the reconstruction formula \eqref{rec_low} starting from a lower semi-frame. However, our aim is to extend \eqref{rec_low} to the generic case, i.e. not assuming the density of $\D(C_\xi)$, giving the answer to Question \ref{quest1}.

Let $\xi$ be a lower semi-frame of $\H$. Thus  $T_\xi:\D(T_\xi)\subseteq \H_\xi \to \H_\xi$ is a bijection and $T_\xi^{-1}:\H_\xi \to \H_\xi$ is bounded (also its positive square root $T_\xi^{-\mez}:\H_\xi \to \H_\xi$ is bounded). In what follows we will need the projection $P$ of $\H$ onto $\H_\xi$. 
Let $h\in \H$, then for all $g\in \D(C_\xi)\subseteq \H_\xi$  
\begin{align}
	\label{weak_rec_1seq(a)}
	\pin{h}{g}&=\pin{Ph}{g}=\pin{T_\xi T_\xi ^{-1}Ph}{g}=\sum_{n\in \N} \pin{T_\xi^{-1}Ph}{\xi_n}\pin{\xi_n}{g}=\sum_{n\in \N} \pin{T_\xi^{-1}Ph}{P\xi_n}\pin{\xi_n}{g} \nonumber\\
	&=\sum_{n\in \N} \pin{Ph}{T_\xi^{-1}P\xi_n}\pin{\xi_n}{g}=\sum_{n\in \N} \pin{h}{T_\xi^{-1}P\xi_n}\pin{\xi_n}{g}.
\end{align}
Note that $\{T_\xi^{-1}P\xi_n\}_{n\in \N}$ is a sequence in $\H_\xi$, but we consider it also as a sequence of $\H$. Moreover, $\{T_\xi^{-1}P\xi_n\}_{n\in \N}$
is a Bessel sequence of $\H$. Indeed, for every $f\in \H$, we have that $T_\xi^{-1}Pf\in \D(C_\xi)$ and
\begin{equation}
	\label{dual_Bess}
	\sum_{n\in \N} |\pin{f}{T_\xi^{-1}P\xi_n}|^2=
	\sum_{n\in \N} |\pin{T_\xi^{-1} Pf}{\xi_n}|^2= \|T_\xi^{\mez} T_\xi^{-1}Pf\|^2
	\leq
	\|T_\xi^{-\mez}\|^2\|f\|^2. 
\end{equation}
Hence, from  \eqref{weak_rec_1seq(a)} we get the following reconstruction in a strong sense.
\begin{theo}
	Let $\xi$ be a lower semi-frame of $\H$ with generalized frame operator $T_\xi$ and let $P$ be the projection of $\H$ onto $\H_\xi$. Then $\{T_\xi^{-1}P\xi_n\}_{n\in \N}$
	is a Bessel sequence of $\H$ and 
	\begin{equation}
		\label{weak_rec_1seq(b)}
		g=\sum_{n\in \N} \pin{g}{\xi_n} T_\xi^{-1}P\xi_n, \qquad \forall g\in \D(C_\xi),
	\end{equation}
	with unconditional convergence. 
\end{theo}

If $\D(C_\xi)$ is dense, then $\H_\xi=\H$ so $T_\xi:\D(T_\xi)\to \H$ is bijective and the sequence $\{T_\xi^{-1}P\xi_n\}_{n\in \N}$ is simply $\{T_\xi^{-1}\xi_n\}_{n\in \N}$ 
which recalls the expression of the canonical dual of a frame (where $S_\xi$ is now replaced with $T_\xi$). For this reason we extend the terminology and call  $\{T_\xi^{-1}P\xi_n\}_{n\in \N}$ the {\it canonical dual} of the lower semi-frame $\xi$. 

Actually, a formula like \eqref{weak_rec_1seq(b)} involving a lower semi-frame and a Bessel sequence was proved in \cite[Proposition 3.4]{Casazza_lower}: given a lower semi-frame $\xi$ of $\H$ there exists a Bessel sequence $\eta=\{\eta_n\}_{n\in \N}$ of $\H$ such that 
\begin{equation*}
	g=\sum_{n\in \N} \pin{g}{\xi_n} \eta_n, \qquad \forall g\in \D(C_\xi).
\end{equation*}
The proof in \cite[Proposition 3.4]{Casazza_lower} is constructive and, more precisely, $\eta$ is given as follows. The operator $C_\xi^{-1}:R(C_\xi)\to \H_\xi$ is well-defined and bounded; therefore it can be extended to a bounded operator $Y:\ell^2(\N)\to \H_\xi$ by putting $Y\{c_n\}=0$ for $\{c_n\}\in R(C_\xi)^\perp$, i.e. $Y=C_\xi^{-1} \oplus 0$. The sequence $\eta$ is now defined as $\eta_n:=Ye_n$ for all $n\in \N$, where $\{e_n\}_{n\in \N}$ is the canonical orthonormal basis of $\ell^2(\N)$. 

Thus we come to Question \ref{quest2}: we prove in particular that the Bessel sequences $\{\eta_n\}_{n\in \N}$ and $\{T_\xi^{-1}P\xi_n\}_{n\in \N}$ coincide. 

\begin{pro}
	Let $\xi$ be a lower semi-frame of $\H$ with generalized frame operator $T_\xi$, and $\eta$ as before. Then $\eta_n=T_\xi^{-1}P\xi_n$ for every $n\in \N$.
\end{pro}
\begin{proof}
	From $T_\xi=C_\xi^\times C_\xi$ we have $T_\xi^{-1}=YX$ where $Y$ is defined as above and $X:\H_\xi \to \ell^2(\N)$ is the bounded pseudo-inverse of $C_\xi^\times$. 
	If $\{e_n\}_{n\in \N}$ is the canonical orthonormal basis of $\ell^2(\N)$, then 
	$\pin{C_\xi f}{e_n}=\pin{f}{\xi_n}=\pin{f}{P\xi_n}$ for $f\in \D(C_\xi)$. This means that $e_n\in \D(C_\xi^\times)$ and $C_\xi^\times e_n=P\xi_n$ for all $n\in \N$. By definition of pseudo-inverse, $XC_\xi^\times\{c_n\}=\{c_n\}$ for all $\{c_n\}\in \D(C_\xi^\times)$.
	In particular, $e_n=XC_\xi^\times e_n=XP \xi_n$ and, as consequence, $T_\xi^{-1}P\xi_n=Ye_n=\eta_n$ for all $n\in \N$. 
\end{proof}

\begin{rem}
	\label{rem_cntrex_lower}
	In general, \eqref{weak_rec_1seq(b)} does not always extend to all $g\in \H$ as it was shown first in \cite[Theorem 3.5]{Casazza_lower} and then also in \cite[Example 4.1]{Stoeva_lower_p}.  We give more details about these results. 
	\begin{enumerate}
		\item The counterexample in \cite[Theorem 3.5]{Casazza_lower} is constructed as a sequence $\xi$ in a direct sum of Hilbert spaces of increasing finite dimensions $\H=\bigoplus_{n\geq 2} \H_n'$. Moreover, $\D(C_\xi)$ is dense and $\n{\xi_n}$ is constant for all $n\in \N$. 
		\item The counterexample  given by \cite[Example 4.1]{Stoeva_lower_p} is simpler, namely, $\xi=\{\xi_n\}_{n\geq 2}$ where $\xi_n=n(e_1+e_n)$ and $\{e_n\}_{n\in \N}$ is an orthonormal basis  of $\H$. 
		By the way, for this lower semi-frame we have $\D(C_\xi)^\perp=\{e_1\}$, i.e. $\D(C_\xi)$ is not dense. Actually, it is sufficient to take $\xi=\{\xi_n\}_{n\geq 2}$ with $\xi_n=e_1+e_n$; $\xi$ is a lower semi-frame and $\D(C_\xi)=\{e_1\}^\perp$. This example can be also generalized to any minimal sequence $\xi$ with $\D(C_\xi)$ not dense: we cannot have \eqref{weak_rec_1seq(b)} for every $f\in \H$, otherwise $\xi$ would be a Schauder basis and then $\D(C_\xi)$ would be dense. 
	\end{enumerate}		
\end{rem}

We will find and discuss another example of lower semi-frames which does not give expansion in the whole space (Example \ref{countex_transl}).

\begin{exm}
	\label{exm_diana_mod}
	Let us consider the sequence 
	$\xi=\{\xi_n\}_{n\geq 2}$, where $\xi_n=e_1+e_n$ and $\{e_n\}_{n\in \N}$ is an orthonormal basis  of $\H$, that we have introduced before. As already said, $\xi=\{\xi_n\}_{n\geq 2}$ is a lower semi-frame and $\D(C_\xi)=\{e_1\}^\perp$. Our aim here is to find the canonical dual $\{T_\xi^{-1}P\xi_n\}_{n\in \N}$. First of all 
	$$
	\t_\xi(f,g)=\sum_{n=2}^\infty \pin{f}{\xi_n}\pin{\xi_n}{g}=\pin{f}{g}, \qquad\forall f,g\in \D(C_\xi),
	$$
	thus $T_\xi f =f$ for $f\in \D(T_\xi)=\D(C_\xi)=\{e_1\}^\perp$. Now, denoting with $P$ the projector onto $\{e_1\}^\perp$, we have $T_\xi^{-1}P\xi_n=T_\xi^{-1} e_n=e_n$   
	for every $n\geq 2$. \\
	Note that the statements in \cite[Proposition 3.11]{Anto_Bal_1} and in  \cite[Proposition 3.11]{Anto_Bal} are incorrect because a lower semi-frame, like the one in this example, may not have upper semi-frames for the whole space (i.e. complete Bessel sequences) as duals. The reconstruction formula is only guaranteed in $\D(C_\xi)$. 
\end{exm}

We mention that in \cite{Anto_Bal_1,Anto_Bal} a formula like \eqref{weak_rec_1seq(a)} has been given starting from an upper semi-frame instead of a lower semi-frame. There is a crucial difference in the two approaches. Namely, for an upper semi-frame $\xi$, the operator $T_\xi$ is injective, bounded (so that $T_\xi=S_\xi$) but its range $R(T_\xi)$ is a proper subspace of $\H$, unless $\xi$ is a frame. Therefore, $\{T_\xi^{-1}\xi_n\}_{n\in \N}$ might be not well-defined, as shown in \cite[Section 2.6]{Anto_Bal_1} and also in Remark \ref{rem_upp_trans}.

To conclude this section, if $S_\xi=T_\xi$, then we obtain also the following  reconstruction formula in a strong sense
\begin{align}
	\label{eq_6}
	f&=T_\xi T_\xi^{-1}f=\sum_{n\in \N}\pin{T_\xi^{-1} f}{\xi_n}\xi_n=\sum_{n\in \N}\pin{f}{T_\xi^{-1}P\xi_n}\xi_n, \qquad \forall f\in\ol{\D(C_\xi)}.
\end{align}
Note that in \eqref{eq_6} and in \eqref{weak_rec_1seq(b)} the sequences $\xi$ and $\{T_\xi^{-1}P\xi_n\}_{n\in \N}$ are in different places.  
Moreover, we can state a characterization when the synthesis operator is closed. 

\begin{pro}
	Let $\xi=\{\xi_n\}_{n\in \N}$ be a sequence of $\H$ with closed synthesis operator $D_\xi$. The following statements are equivalent.
	\begin{enumerate}
		\item $\xi$ is a lower semi-frame;
		\item there exists a Bessel sequence $\{\eta_n\}_{n\in \N}$ of $\H$ such that 
		\begin{equation}
			\label{eq_7}
			f=\sum_{n\in \N} \pin{f}{\eta_n}\xi_n, \qquad\forall f\in \H;
		\end{equation}
		\item $\mathcal{R}(D_\xi)=\H$.		
	\end{enumerate}
\end{pro}
\begin{proof}
	(i) $\implies$ (ii) Since $D_\xi$ is closed, by Proposition \ref{pro_oper_CD},  $\D(C_\xi)$ is dense and $S_\xi=D_\xi C_\xi=C_\xi^*C_\xi=T_\xi$. Hence,  \eqref{eq_6} holds for every $f\in \H$. Moreover, as we have seen before, $\{T_\xi^{-1}\xi_n\}$ is a Bessel sequence.\\
	(ii) $\implies$  (iii)  Since $\{\pin{f}{\eta_n}\}_{n\in \N}\in \ell^2(\N)$ for every $f\in \H$, \eqref{eq_7} implies that $\mathcal{R}(D_\xi)=\H$.\\
	(iii) $\implies$  (i) This implication was originally proved in \cite[Theorem 4.1]{Ole_pseudo} and also in a different way in the discussion after \cite[Theorem 5.3]{Stoeva_arxiv}. 
\end{proof}

We can actually state one more property of $T_\xi$ making a consideration analogous to \eqref{dual_Bess}. 

\begin{rem}
	Let $\xi$ be a lower semi-frame of $\H$ with generalized frame operator $T_\xi$ and let $P$ be the projection of $\H$ onto $\H_\xi$. Then, 	$T_\xi^{-1/2}P\xi$ is a {\it Parseval frame} for $\H_\xi$ (i.e. it has frame bounds $A=B=1$) and therefore 
	$$
	f=\sum_{n\in I} \pin{f}{T_\xi^{-\mez}P\xi_n}T_\xi^{-\mez}P\xi_n, \qquad \forall f\in \ol{\D(C_\xi)}.
	$$
	When $\D(C_\xi)$ is dense, then $T_\xi^{-1/2}\xi$ is a Parseval frame for $\H$ (a well-known fact for a frame $\xi$, in which case $T_\xi^{-1/2}\xi$ is called the {\it canonical tight frame} of $\xi$). The considerations in this remark have been further investigated in \cite{ACT}.
\end{rem}

\section{Relations between $T_\xi$, $S_\xi$ and $W_\xi$}
\label{sec_TW}

In this section we are going to give a negative answer to Question \ref{quest3}. In \cite[Example 1]{Corso_seq} it was shown a sequence $\xi$ such that $\D(C_\xi)$ is dense and $T_\xi$ is different from $S_\xi$ (moreover it is possible to choose a lower semi-frame for $\xi$, see \cite[Example 2]{Corso_seq}). We will encounter in Example \ref{ST_translates} another case in which these operators are different in the context of sequences of translates.

When $\D(C_\xi)$ is dense then $W_\xi\subseteq T_\xi$ also holds, where $W_\xi$ is the {\it weak frame operator} introduced in \eqref{def_W}. However, we construct here a sequence $\xi$ such that $\D(C_\xi)$ is dense and $T_\xi$ is strictly larger than $W_\xi$.

\begin{exm}
	\label{count_exm_TW_xi} 
	Let $\{e_n\}_{n\in \N}$ be an orthonormal basis  of $\H$. 
	Let $\xi=\{\xi_n\}_{n\in \N}$, $\eta=\{\eta_n\}_{n\in \N}$ be two sequences of $\H$ such that $\xi_1=e_1$ and $ \xi_n=n^{\frac{8}{5}}(e_n-e_{n-1})$ for $n\geq 2$ 
	and $\eta_n=n^{\frac{1}{2}}e_n$ for $n\in \N$.  Given $f\in \H$ we denote by $f_n:=\pin{f}{e_n}$. We have
	$$\D(C_\xi)=\left\{f\in \H: \sum_{n=2}^\infty n^\frac{16}{5}|f_n-f_{n-1}|^2< \infty\right\},$$ 
	$$\D(C_\eta)=\left\{f\in \H: \sum_{n=1}^\infty n|f_n|^2< \infty\right\}. $$  
	We define the new sequence $\chi:=\{\xi_1,\eta_1,\xi_2,\eta_2,\dots\}$. The analysis operator $C_\chi$ has domain $\D(C_\chi)= \D(C_\xi)\cap \D(C_\eta)$, which is dense because it contains $\{e_n\}_{n\in \N}$.
	We are going to show that $W_{\chi}\subsetneq T_{\chi}$. First of all, for $f\in \H$ and $k\geq 2$ we find that
	\begin{align}
		\sum_{n=1}^{2k} \pin{f}{\chi_n}\chi_n &=\sum_{n=1}^{k} \pin{f}{\xi_n}\xi_n+\sum_{n=1}^{k} \pin{f}{\eta_n}\eta_n \nonumber\\
		&=(2f_1-2^{\frac{16}{5}}f_2+2^{\frac{16}{5}}f_1)e_1\nonumber\\
		&+\sum_{n=2}^{k-1} (n^\frac{16}{5}f_n-n^\frac{16}{5}f_{n-1}-(n+1)^\frac{16}{5}f_{n+1}+(n+1)^\frac{16}{5}f_n+nf_{n})e_n\nonumber\\
		&+(k^\frac{16}{5}f_k-k^\frac{16}{5}f_{k-1}+kf_k)e_k, \label{eq_sum1}
	\end{align}	
	and 
	\begin{align}
		\sum_{n=1}^{2k-1} \pin{f}{\chi_n}\chi_n &=\sum_{n=1}^{k} \pin{f}{\xi_n}\xi_n+\sum_{n=1}^{k-1} \pin{f}{\eta_n}\eta_n \nonumber\\
		&=(2f_1-2^{\frac{16}{5}}f_2+2^{\frac{16}{5}}f_1)e_1\nonumber\\
		&+\sum_{n=2}^{k-1} (n^\frac{16}{5}f_n-n^\frac{16}{5}f_{n-1}-(n+1)^\frac{16}{5}f_{n+1}+(n+1)^\frac{16}{5}f_n+nf_{n})e_n \nonumber\\
		&+(k^\frac{16}{5}f_k-k^\frac{16}{5}f_{k-1})e_k. \label{eq_sum2}
	\end{align}	
	In the rest of the example we choose $h\in \H$ such that $h_n=\pin{h}{e_n}=n^{-2}$ for all $n\in \N$. We have 
	\begin{equation}
		\label{eq_2}
		n^{\frac{16}{5}}\left|h_n-h_{n-1}\right|^2
		=4n^{-\frac{14}{5}}+o(n^{-\frac{14}{5}})
	\end{equation}
	where we write $o(n^a)$ for the usual notation of a function $\N\to \C$ such that $o(n^a)/n^a\to 0$ for $n\to \infty$. Thus, by \eqref{eq_2}, $h\in \D(C_{\chi})$. 	For $k\geq 2$ we write in a compact way $\sum_{n=1}^{2k} \pin{h}{\chi_n}\chi_n =(2-2^{-\frac{4}{5}}+2^{\frac{16}{5}})e_1+\sum_{n=1}^{k-1} \alpha_n e_n+\beta_k e_k$ and $\sum_{n=1}^{2k-1} \pin{h}{\chi_n}\chi_n =(2-2^{-\frac{4}{5}}+2^{\frac{16}{5}})e_1+\sum_{n=1}^{k-1} \alpha_n e_n+\gamma_k e_k$. From \eqref{eq_sum1} and \eqref{eq_sum2}, straightforward calculations lead to $\alpha_n=\frac{2}{5}n^{-\frac{4}{5}}+o(n^{-\frac{4}{5}})$, $\beta_k=-2k^\frac{1}{5}+o(k^\frac{1}{5})$ and also $\gamma_k=-2k^\frac{1}{5}+o(k^\frac{1}{5})$. 
	Now let $g\in \D(C_{\chi})$ and write again $g_n:=\pin{g}{e_n}$. The series  $\sum_{n=1}^{2k} \pin{h}{\chi_n}\pin{\chi_n}{g}=(2-2^{-\frac{4}{5}}+2^{\frac{16}{5}})\ol{g_1}+\sum_{n=2}^{k-1} \alpha_n \ol{g_n}+\beta_k \ol{g_k}$  and
	$\sum_{n=1}^{2k-1} \pin{h}{\chi_n}\pin{\chi_n}{g}=(2-2^{-\frac{4}{5}}+2^{\frac{16}{5}})\ol{g_1}+\sum_{n=2}^{k-1} \alpha_n \ol{g_n}+\gamma_k \ol{g_k}$
	are convergent for $k\to \infty$ and convergent to the same limit. Indeed, $\{\alpha_n\}_{n\geq 2}$ belongs to $\ell^2(\N)$ and, for sufficiently large values of $k$, $|\beta_k||g_k|,|\gamma_k||g_k|\approx 2k^{\frac{1}{5}}|g_k|\leq 2k^{\frac{1}{2}}|g_k|\to 0$, since $g$ is in particular in $\D(C_\eta)$. Moreover, $$\left |\sum_{n=1}^\infty  \pin{h}{\chi_n}\pin{\chi_n}{g}\right |\leq C\|g\|, \qquad \forall g\in \D(C_\chi),$$
	for some $C>0$.  We thus conclude that $h\in \D(T_{\chi})$. 
	
	Anyway, 
	$\sum_{n=1}^{k} \pin{h}{\chi_n}\chi_n$ cannot converge weakly for $k\to \infty$ on $\H$, because its norm goes to infinity for $k\to \infty$. In other words, $h\notin \D(W_{\chi})$.
\end{exm}

The operator $W_\xi$ may be of independent interest, thus we end this section spending some more words on it. 
For completeness we also define the operator
$Q_\xi:\D(Q_\xi)\subseteq \ell^2(\N) \to \H$  on the domain
$$\D(Q_\xi):=\left \{ \{c_n\}\in \ell^2(\N):\sum_{n\in \N} c_n \xi_n \text{ is weakly convergent in } \H\right \}$$ as
$Q_\xi\{c_n\}=\sum_{n\in \N} c_n \xi_n$ in weak sense. 
Some properties of $Q_\xi$ and $S_\xi$ are given below and can be proved as done in \cite[Proposition 3.3]{Classif} for $D_\xi$ and $S_\xi$.

\begin{pro}
	\label{pro_QW_xi}
	Let $\xi=\{\xi_n\}_{n\in \N}$ be a sequence of $\H$. The following statements hold.
	\begin{enumerate}
		\item $Q_\xi$ is densely defined and  $D_\xi\subseteq Q_\xi$ and $Q_\xi^*=C_\xi$.
		\item If $\D(C_\xi)$ is dense, then $Q_\xi\subseteq C_\xi^*$.
		\item $Q_\xi$ is closable if and only if $\D(C_\xi)$ is dense if and only if $D_\xi$ is closable.
		\item $Q_\xi$ is closed  if and only if  $\D(C_\xi)$ is dense and $C_\xi^*=Q_\xi$. 
		\item $S_\xi\subseteq W_\xi$ and $W_\xi=Q_\xi C_\xi$. 
	\end{enumerate}
\end{pro}

An application of Banach-Steinhaus theorem shows that $\xi$ is a Bessel sequence if and only if $\D(Q_\xi)=\ell^2(\N)$ (in that case $Q_\xi=D_\xi$) if and only if $\D(W_\xi)=\H$ (in that case $W_\xi=S_\xi$). 
Sufficient conditions for which $T_\xi=S_\xi$ or  $T_\xi=W_\xi$ are given in the following propositions. 

\begin{pro}
	\label{pro_cs_S=T}
	Let $\xi=\{\xi_n\}_{n\in \N}$ be a sequence of $\H$ such that $\D(C_\xi)$ is dense.  The following statements hold.
	\begin{enumerate}
		\item If $D_\xi$ is closed, then $T_\xi=S_\xi$.
		\item If $Q_\xi$ is closed, then  $T_\xi=W_\xi$.
		\item If $f\in \D(T_\xi)$, then $f\in \D(W_\xi)$ if and only if $\sup_k \n{\sum_{n=1 }^k \pin{f}{\xi_n}\xi_n}<\infty$.   
	\end{enumerate}
\end{pro}
\begin{proof}
	Points (i-ii) are consequences of Propositions \ref{pro_oper_CD} and \ref{pro_QW_xi}: in particular, if $D_\xi$ is closed, then $T_\xi=C_\xi^*C_\xi=D_\xi C_\xi=S_\xi$. For (iii), let us remember that $W_\xi  \subseteq T_\xi$. Take first $f\in \D(W_\xi)$; then $\{\sum_{n=1 }^k \pin{f}{\xi_n}\xi_n\}_{k\in \N}$ is norm bounded because it is weakly convergent (\cite[Sect. III.1.6]{Kato}). Now let $f\in \D(T_\xi)$ be such that 
	$M=\sup_k \n{\sum_{n=1 }^k \pin{f}{\xi_n}\xi_n}<\infty$. We suppose $M>0$, otherwise the statement is trivial. For $h\in \H$ and $\epsilon>0$, there exist $g\in \D(C_\xi)$ such that $\n{h-g}<\frac{\epsilon}{4M}$ and $N\in \N$ such that $|\pin{\sum_{n=k}^m \pin{f}{\xi_n}\xi_n }{g}|<\frac{\epsilon}{2}$ for all $m\geq k\geq N$. Therefore
	\begin{align*}
		\left |\left\langle\sum_{n=k}^m \pin{f}{\xi_n}\xi_n,h\right\rangle\right |\leq&\left |\left\langle\sum_{n=k}^m \pin{f}{\xi_n}\xi_n ,g\right\rangle\right |+\left |\left\langle\sum_{n=k}^m \pin{f}{\xi_n}\xi_n ,h-g\right\rangle\right | \\
		\leq& \frac{\epsilon}{2}+\left \|\sum_{n=k}^m \pin{f}{\xi_n}\xi_n \right \|\|h-g\|< \frac{\epsilon}{2}+2M\frac{\epsilon}{4M}=\epsilon.
	\end{align*}
	Hence $f\in \D(T_\xi)$. 
\end{proof}

\begin{pro}
	\label{pro_cs_S=T_2}
	If $\xi$ is biorthogonal to a complete sequence and it is a lower semi-frame of $\H$, then $S_\xi=T_\xi$ if and only if $\xi$ is a Schauder basis of $\H$.
\end{pro}
\begin{proof}
	We note that $\D(C_\xi)$ is dense since the unique complete sequence $\widetilde {\xi}:=\{\widetilde {\xi_n}\}_{n\in I}$ biorthogonal to $\xi$ is contained in $\D(C_\xi)$.  Moreover $\widetilde {\xi_n}\in \D(S_\xi)$ and $S_\xi\widetilde{\xi_n}=\xi_n$ for all $n\in \N$. 
	Suppose that $S_\xi=T_\xi$. Then, by Proposition  \ref{car_seq_T_xi}, $0\in \rho(S_\xi)$, i.e. for every $f\in \H$
	$$f=S_\xi S_\xi^{-1}f=\sum_{n\in \N}\pin{S_\xi^{-1}f}{\xi_n}\xi_n=\sum_{n\in \N}\pin{f}{S_\xi^{-1}\xi_n}\xi_n.$$
	Since $\xi$ is minimal, this equality means that $\xi$ is a Schauder basis of $\H$. 
	Now suppose 
	that $\xi$ is a Schauder basis of $\H$ and let
	$h\in \D(T_\xi)$. Thus
	$$
	T_\xi h=\sum_{n\in \N}\pin{T_\xi h}{\widetilde {\xi_n}}\xi_n=\sum_{n\in \N}\pin{h}{T_\xi \widetilde {\xi_n}} \xi_n=\sum_{n\in \N}\left \langle h,\xi_n\right\rangle \xi_n.
	$$
	Therefore $h\in \D(S_\xi)$. 	
\end{proof}

\section{Sequences of translates}
\label{secSeq1_transl}

Let $\varphi \in L^2(\R)$. The {\it sequence of translates} of $\varphi$ with a parameter $a>0$ is $\mathcal{T}(\varphi,a):=\{\varphi_n\}_{n\in \Z}$ where $\varphi_n(x):=\varphi(x-na)$ for all $x\in \R$ (note that we are considering {\it uniform} translates, i.e. by $na$ for $n\in \Z$).   Our study will follow the line of the works \cite{Janssen} and \cite{Ben_Li,hsww,Ron_Shen} (for an introduction to the topics one can refer also to \cite{Chris,heil}). Indeed, first of all we want to find the analysis operator $C_\varphi$ and the generalized frame operator $T_\varphi$ of $\mathcal{T}(\varphi,a)$ (we omit the dependence on $a$ in the symbols $C_\varphi,T_\varphi$). Secondly, we want to characterize those sequences $\mathcal{T}(\varphi,a)$ which are lower semi-frames of their closed spans. 

We recall that the operator $T_\varphi$ is unconditionally defined; this means that $T_\varphi$ does not change if we order $\mathcal{T}(\varphi,a)$ as any sequence indexed by $\N$.  We mention that in the context of sequences of translates the index $\Z$ is sometimes ordered as $\Z=\{0,1,-1,2,-2,\dots\}$ (see for instance \cite{nielsikic}).

As we are going to see, the results are formulated in terms of the Fourier transform $\mathcal{F}$  defined in $L^1(\R)$ as 
$$
(\mathcal{F} f) (\gamma):=\wh f(\gamma):=\int_\R f(x) e^{-2\pi i x \gamma} dx, \qquad \gamma \in \R. 
$$
The Fourier transform defines a unitary operator $L^2(\R)\to L^2(\R)$ in a standard way, which we still denote by $\mathcal{F}$.  A remarkable property of $\mathcal{F}$ concerns  translations with modulations which are defined as $(\mathcal{T}_{y}f)(x)=f(x-y)$ and $(\mathcal{M}_\omega f) (x)=e^{2\pi i \omega x}f(x)$ for $x\in \R$ and $y,\omega\in \R$, respectively. More precisely (see e.g. \cite[Ch. 2]{Chris}), the Fourier transform interchanges translations with modulations, i.e. $\mathcal{F}\mathcal{T}_y=\mathcal{M}_{-y}\mathcal{F}$ and $\mathcal{F}\mathcal{M}_y=\mathcal{T}_{y}\mathcal{F}$. So we have in particular
$$\wh{\varphi_n}(\gamma)=\wh{\varphi}(\gamma)e^{-2 \pi i na \gamma}, \qquad \forall\gamma\in \R.$$

When $\mathcal{T}(\varphi,a)$ is a Bessel sequence, then $T_\varphi$ is the frame operator $S_\varphi$ and we already know the expression (in the Fourier domain) of $S_\varphi$, called the {\it Walnut representation}  
(see \cite[Theorem 10.2.1]{Chris}):
$$
\wh {S_\varphi f}(\gamma)=\frac{1}{a} \wh{\varphi}(\gamma)\sum_{n\in \Z} \wh \varphi \left (\gamma-\frac{n}{a}\right)\wh f\left(\gamma-\frac{n}{a}\right), \qquad \gamma\in \R.
$$ 
\no 
The expression of $T_\varphi$ when $\mathcal{T}(\varphi,a)$ is not a Bessel sequence, that we will find in  Theorem \ref{th_shift}, is the same but of course the domain is smaller than $L^2(\R)$. 

We now define the function 
$p_\varphi(\gamma):=\frac{1}{a}\sum_{n\in \Z}\left|\wh \varphi \left(\frac{\gamma+n}{a}\right) \right|^2$ (the definition of $p_\varphi$ here is slightly different than the one in \cite[Section 9.2]{Chris} because we include the factor $\frac{1}{a}$). We recall some results from \cite{Ron_Shen,hsww,Chris,heil}: $p_\varphi\in L^1(0,1)$, $\mathcal{T}(\varphi,a)$ is a Bessel sequence if and only if $p_\varphi$ is bounded a.e. in $[0,1]$ and that $\mathcal{T}(\varphi,a)$ is a frame for its closed span if and only if $p_\varphi$ is bounded above and below away from zero a.e. in $Z_\varphi$, where $Z_\varphi:=\{\gamma\in [0,1):p_\varphi(\gamma)>0\}$. The properties that  $\mathcal{T}(\varphi,a)$ is an orthonormal or a Riesz sequence are likewise characterized in terms of $p_\varphi$. 
We mention that $p_\varphi$ describes also various levels of linear independence (see \cite{Saliani,SikicSlamic,Slamic}). 

We also remark that some results have been extended to irregular frames of translates in \cite{BH} with an operator based approach. Moreover, the function that plays a central rule in \cite{BH} is $|\wh \varphi|^2$.

To prove our results about $C_\varphi$ and $T_\varphi$ we will make use of the following lemma. This is a more general version of \cite[Lemma 9.2.4]{Chris} and of \cite[Lemma 1.4.1]{Groechenig_b} (indeed, we do not assume that $f$ is bounded). 

\begin{lem}
	\label{lem_period}
	Let $a>0$ be given and $f,g:\R\to \C$ measurable functions such that $f$ is $a$-periodic. Then  $fg\in L^1(\R)$ if and only if $f(\cdot) \sum_{k\in \Z} |g(\cdot-ka)|\in L^1(0,a)$. Moreover, under these conditions,
	$$
	\int_{-\infty}^\infty f(x)g(x)dx=\int_0^a f(x) \sum_{k\in \Z} g(x-ka) dx.
	$$
\end{lem}
\begin{proof}
	First  we note that
	\begin{align*}
		\int_{-\infty}^\infty |f(x)g(x)|dx &= \sum_{k\in \Z}\int_0^a |f(x-ka)||g(x-ka)|dx\\
		&=\sum_{k\in \Z}\int_0^a |f(x)||g(x-ka)|dx.
	\end{align*}
	Now Tonelli's theorem ensures that 
	\begin{align*}
		\int_{-\infty}^\infty |f(x)g(x)|dx &=\int_0^a |f(x)|\sum_{k\in \Z} |g(x-ka)|dx. 
	\end{align*}
	In other words, $fg\in L^1(\R)$ if and only if $f(\cdot) \sum_{k\in \Z} |g(\cdot-ka)|\in L^1(0,a)$. \\
	The second part of the statement is obtained by Lebesgue dominated convergence theorem.
\end{proof}

Finally, in our study the {\it bracket product}
\begin{equation*}
	\left [\wh f,{\wh \varphi}\right](\gamma):=\frac{1}{a}\sum_{n\in \Z}\wh f\left(\frac{\gamma-n}{a}\right)\ol{\widehat{\varphi}\left (\frac{\gamma-n}{a}\right)}, \qquad \gamma\in [0,1]
\end{equation*}
will be particularly useful. The bracket product was first used in \cite{JM,BVR,BVR2}. 
Moreover 
$\left [\wh f,{\wh \varphi}\right]$ is a well-defined element of $L^1(0,1)$  for $f,\varphi\in L^2(\R)$, indeed $\wh f\wh \varphi \in L^1(\R)$ and 
\begin{align*}
	\int_0^1 \left|\sum_{n\in \Z}\wh f\left(\frac{\gamma-n}{a}\right)\ol{\widehat{\varphi}\left (\frac{\gamma-n}{a}\right)}\right| d\gamma &\leq \int_0^1 \sum_{n\in \Z}\left|\wh f\left(\frac{\gamma-n}{a}\right)\ol{\widehat{\varphi}\left (\frac{\gamma-n}{a}\right)}\right| d\gamma  \\
	&= \int_{-\infty}^\infty \left|\wh f\left (\frac{\gamma}{a}\right)\ol{\wh \varphi\left (\frac{\gamma}{a}\right)}\right|d\gamma< \infty
\end{align*}
where the last equality is by Lemma \ref{lem_period}. 
We define the operator $\Gamma_\varphi:\D(\Gamma_\varphi)\subseteq L^2(\R)\to L^2(0,1)$ where
\begin{equation}
	\label{Gamma}
	\D(\Gamma_\varphi)=\left \{f\in L^2(\R):\left [\wh f,{\wh \varphi}\right]\in L^2(0,1) \right \} 
	\quad\text{ and }\quad \Gamma_\varphi f=\left [\wh f,{\wh \varphi}\right].
\end{equation}

We are now ready to determine the operators $C_\varphi$ and $T_\varphi$. We write $e_n(\gamma)=e^{2 \pi i n \gamma}$ for any $n\in \Z$ and $\gamma\in (0,1)$.

\begin{theo}
	\label{th_shift}
	Let $\varphi \in L^2(\R)$ and $a>0$. Let $C_\varphi$ and $T_\varphi$ be the analysis and the generalized frame operators of $\mathcal{T}(\varphi,a)$, respectively. The following statements hold.
	\begin{enumerate}
		\item The operator $C_\varphi$ has domain $\D(C_\varphi)=\D(\Gamma_\varphi)$ and $C_\varphi f=\{\pin{ [\wh f,{\wh \varphi}]}{e_{-n}}\}_{n\in \Z}$ for all $f\in \D(C_\varphi)$. The domain $\D(C_\varphi)$ contains all functions whose Fourier transform is bounded and compactly supported, thus $\D(C_\varphi)$ is dense.
		\item The adjoint $C_\varphi^*$ has domain $$\D(C_\varphi^*)=\left \{\{c_n\}\in \ell^2(\Z): \int_{-\infty}^{\infty}\left|\sum_{n\in \Z}c_ne^{2\pi i na\gamma}\wh \varphi(\gamma)\right|^2d\gamma <\infty\right \}$$  and it is given by $\wh {C_\varphi^*\{c_n\}}(\gamma)=\sum_{n\in \Z}c_ne^{2\pi i na\gamma}\wh \varphi(\gamma)$ for $\{c_n\}\in \D(C_\varphi^*)$ and a.e. $\gamma\in \R$.
		\item The domain $\D(T_\varphi)$ of $T_\varphi$ is the subspace of  $f\in L^2(\R)$ such that
		\begin{enumerate}
			\item[\emph{(a)}] $\left [\wh f,{\wh \varphi}\right]\in L^2(0,1)$;
			\item[\emph{(b)}] 
			the function from $\R$ to $\C$ given by $\gamma\mapsto\widehat{\varphi} (\gamma)\sum_{n\in \Z}\wh f\left(\gamma-\frac{n}{a}\right)\ol{\widehat{\varphi}\left (\gamma-\frac{n}{a}\right)}$ belongs to $L^2(\R)$,
		\end{enumerate}
		and 
		$$\wh {T_\varphi f}(\gamma)=\frac{1}{a}\widehat{\varphi} (\gamma)\sum_{n\in \Z}\wh f\left(\gamma-\frac{n}{a}\right)\ol{\widehat{\varphi}\left (\gamma-\frac{n}{a}\right)}
		,$$
		for $f\in \D(T_\varphi)$ and a.e. $\gamma \in \R$. 
	\end{enumerate}
\end{theo}
\begin{proof}
	\begin{enumerate}
		\item[(i)] 
		We have by Lemma \ref{lem_period} 
		\begin{align*}
			\pin{f}{\varphi_n}=\pin{\widehat f}{\wh{\varphi_n}}&=\int_{-\infty}^\infty \wh f(\gamma)\ol{\wh\varphi(\gamma)}e^{2 \pi i n a\gamma}d\gamma\\
			&
			=\frac{1}{a}\int_{-\infty}^\infty \wh f\left (\frac{\gamma}{a}\right)\ol{\wh\varphi\left (\frac{\gamma}{a}\right)}e^{2 \pi i n\gamma}d\gamma\\
			&=\frac{1}{a}\int_0^1 \sum_{n\in \Z}\wh f\left(\frac{\gamma-n}{a}\right)\ol{\widehat{\varphi}\left (\frac{\gamma-n}{a}\right)}e^{2 \pi i n \gamma} d\gamma\\
			&=\left\langle\left [\wh f,{\wh \varphi}\right],e_{-n} \right\rangle.
		\end{align*} 
		Therefore $\{\pin{f}{\varphi_n}\}\in \ell^2(\Z)$ if and only if $\left [\wh f,{\wh \varphi}\right] \in L^2(0,1)$.  The second part of the statement follows now easily.		 
		
		\item[(ii)] We have $C_\varphi=C_{\{e_n\}}\Gamma_\varphi$, where $C_{\{e_n\}}:L^2(0,1)\to \ell^2(\Z)$ is the analysis operator of $\{e_n\}$. Since $C_\varphi^*=\Gamma_\varphi^*C_{\{e_n\}}^*=\Gamma_\varphi^*C_{\{e_n\}}^{-1}$, the statement follows immediately after we determine $\Gamma_\varphi^*$. To do this, let $h\in \D(\Gamma_\varphi^*)\subseteq L^2(0,1)$, then for any $f\in \D(\Gamma_\varphi)$ 
		$$
		\pin{\Gamma_\varphi^* h}{f}=\pin{h}{\Gamma_\varphi f}=\frac{1}{a}\int_0^1 h(\gamma)\sum_{n\in \Z}\widehat{\varphi}\left (\frac{\gamma-n}{a}\right)\ol{\wh f\left(\frac{\gamma-n}{a}\right)} d\gamma.
		$$
		We write for $\widetilde{h}$ the periodic extension of $h$ on $\R$  and then we can apply Lemma \ref{lem_period} because $h,\Gamma_\varphi f\in L^2(0,1)$, so
		\begin{equation}
			\label{eq1}
			\pin{\Gamma_\varphi^* h}{f}=\frac{1}{a}\int_{-\infty}^\infty \widetilde{h}(\gamma)\widehat{\varphi}\left (\frac{\gamma}{a}\right)\ol{\wh f\left(\frac{\gamma}{a}\right)} d\gamma=\int_{-\infty}^\infty \widetilde{h}(a\gamma)\widehat{\varphi} (\gamma) \ol{\wh f(\gamma)}d\gamma.
		\end{equation}
		Since $\D(\Gamma_\varphi)$ is dense, by \eqref{eq1} it follows that $\widetilde{h}(a\;\!\!\cdot)\widehat{\varphi} (\cdot): \gamma\mapsto \widetilde{h}(a\gamma)\widehat{\varphi} (\gamma)$ is an element of $L^2(\R)$ and that $\wh{\Gamma_\varphi^*h}(\gamma)=\widetilde{h}(a\gamma)\widehat{\varphi} (\gamma)$ for a.e. $\gamma\in \R$. \\
		On the other hand, if $h\in L^2(0,1)$, $\widetilde{h}$ is the periodic extension of $h$ to $\R$ and $\widetilde{h}(a\;\!\cdot)\widehat{\varphi} (\cdot)\in L^2(\R)$, then with the inverse steps of above we find that $h\in \D(\Gamma_\varphi^*)$ and
		$\wh{\Gamma_\varphi^*h}(\gamma)=\widetilde{h}(a\gamma)\widehat{\varphi} (\gamma)$
		for a.e. $\gamma\in \R$. The operator $\Gamma_\varphi^*$ is now determined; summarizing 
		\begin{equation}
			\label{Gamma*}
			\D(\Gamma_\varphi^*)=\{h\in L^2(0,1):\widetilde{h}(a\;\!\cdot)\widehat{\varphi} (\cdot)\in L^2(\R)\} \;\text{ and }\;\wh{\Gamma_\varphi^*h}(\gamma)=\widetilde{h}(a\gamma)\widehat{\varphi} (\gamma) \;\text{ for a.e. }\gamma \in \R.
		\end{equation} 
		Coming back to $C_\varphi^*=\Gamma_\varphi^*C_{\{e_n\}}^*$ we can say that \begin{align*}
			\D(C_\varphi^*)&=\D(\Gamma_\varphi^*C_{\{e_n\}}^*)=\{\{c_n\}\in \ell^2(\Z):C_{\{e_n\}}^* \{c_n\} \in \D(\Gamma_\varphi^*)\}\\
			&=\left \{\{c_n\}\in \ell^2(\Z):\int_{-\infty}^{\infty}|\sum_{n\in \Z}c_ne^{2\pi i na\gamma}\wh \varphi(\gamma)|^2d\gamma <\infty\right \}
		\end{align*}
		and $\wh {C_\varphi^*\{c_n\}}(\gamma)=\wh {\Gamma_\varphi^*C_{\{e_n\}}^*\{c_n\}}(\gamma)=\sum_{n\in \Z}c_ne^{2\pi i na\gamma}\wh \varphi(\gamma)$ for $\{c_n\}\in \D(C_\varphi^*)$ and a.e. $\gamma\in \R$.
		
		\item[(iii)] Since $C_\varphi$ is densely defined we have $C_\varphi^\times=C_\varphi^*$; so  $T_\varphi=C_\varphi^*C_\varphi$, by \eqref{T=C^*C}.  Moreover, since $\n{C_\varphi f}_2=\n{\Gamma_\varphi f}$ for all $f\in \D(C_\varphi)=\D(\Gamma_\varphi)$, by the polarization identity we have 
		$
		\pin{C_\varphi f}{C_\varphi g}_2=\pin{\Gamma_\varphi f}{\Gamma_\varphi g}
		$ for all $f,g\in \D(C_\varphi)=\D(\Gamma_\varphi)$.  It follows that the operators $T_\varphi=C_\varphi^*C_\varphi$ and $\Gamma_\varphi^*\Gamma_\varphi $ are associated to the same sesquilinear form, thus coincide. Hence, the description of $\D(T_\varphi)=\D(C_\varphi^*C_\varphi)= \D(\Gamma_\varphi^*\Gamma_\varphi)= \{f\in \D(\Gamma_\varphi): \Gamma_\varphi f\in \D(\Gamma_\varphi^*)\}$ is simply given by \eqref{Gamma} and by \eqref{Gamma*}. In particular, the function ${h}=\Gamma_\varphi f=\left [\wh f,{\wh \varphi}\right]$ (i.e., the function in (a)) has periodic extension $\widetilde{h}$ on $\R$ given by $\gamma \mapsto\sum_{n\in \Z}\wh f\left(\frac{\gamma-n}{a}\right)\ol{\widehat{\varphi}\left (\frac{\gamma-n}{a}\right)}$ and $\widehat{\Gamma_\varphi^* h}$ is simply 
		$\gamma\mapsto \widehat{\varphi} (\gamma)\widetilde{h}(a\gamma)$ (i.e., the function in (b)). With these observations, for $f\in \D(C_\varphi^*C_\varphi)$
		\[
		\wh{T_\varphi f}(\gamma)= \wh{C_\varphi^*C_\varphi f}(\gamma)=\wh{\Gamma_\varphi^*\Gamma_\varphi f}(\gamma)=\frac{1}{a}\widehat{\varphi} (\gamma)\sum_{n\in \Z}\wh f\left(\gamma-\frac{n}{a}\right)\ol{\widehat{\varphi}\left (\gamma-\frac{n}{a}\right)}, \qquad\forall a.e.\;\gamma \in \R. \qedhere
		\]
	\end{enumerate}
\end{proof}

We are now interested to discuss about lower semi-frames in the context of sequences of translates. By \cite[Exercise 10.18]{heil} for any $\varphi\in L^2(\R)$ and $a>0$ the sequence $\mathcal{T}(\varphi,a)$ is not complete in $L^2(\R)$. Consequently, $\mathcal{T}(\varphi,a)$ is never a lower semi-frame of $L^2(\R)$ (this follows also by the more general result in \cite[Theorem 1.2]{Ole_density}). 
However, we can restrict the space to the closed linear span  $\langle \varphi \rangle:= \ol{\text{span}}$($\{\varphi_n\}_{n\in\Z}$), called as usual the {\it principal shift-invariant subspace}, and study when $\mathcal{T}(\varphi,a)$ is a lower semi-frame of  $\langle \varphi \rangle$.  
It is known (for instance one can see \cite[Theorem 1.6]{hsww} and adapt the result for $a\neq 1$) that there exists a unitary operator $U:L^2([0,1),p_\varphi)\to \langle \varphi \rangle$. Here $L^2([0,1),p_\varphi)$ is the weighted $L^2$-space with measure $p_\varphi(\gamma)d\gamma$; the inner product of $L^2([0,1),p_\varphi)$ is then 
$$\pin{f}{g}_{p_\varphi}=\int_0^1 f(\gamma)\ol{g(\gamma)}p_\varphi(\gamma)d\gamma, \qquad\forall f,g\in L^2([0,1),p_\varphi).$$ 
More precisely,  $Uh= (h\wh \varphi){\;\!\widecheck{}}$, where $\wch{f}$ is the inverse of the Fourier transform of $f\in L^2(\R)$. Therefore
$U^{-1}\varphi_{-n}=e_{n}$, for all $n\in \Z$, where $e_n(\gamma)=e^{2\pi i n \gamma}$, as usual. 
For this reason it is sufficient to analyze the problem in $L^2([0,1),p_\varphi)$ for the sequence $\{e_n\}_{n\in \Z}$. 

\begin{theo}
	\label{pro_lower_transl}
	A sequence of translates $\mathcal{T}(\varphi,a)$ is a lower semi-frame of its closed span if and only if $p_\varphi$ is  bounded away from zero a.e. in $Z_\varphi:=\{\gamma\in [0,1):p_\varphi(\gamma)>0\}$, i.e. there exists $A>0$ such that $p_\varphi(\gamma)\geq A$ for a.e. $\gamma \in Z_\varphi$. \\
	In this case, the canonical Bessel dual of $\mathcal{T}(\varphi,a)$ is $\mathcal{T}(( \wh \varphi \chi_{Z_\varphi}/p_\varphi){\;\!\widecheck{}\;},a)$, with meaning $$ (\wh \varphi\chi_{Z_\varphi}/p_\varphi)(\gamma):=\wh \varphi(\gamma)/p_\varphi(\gamma)\text{ if }\gamma\in Z_\varphi\text{ and } (\wh \varphi\chi_{Z_\varphi}/p_\varphi)(\gamma):=0 \text{ if }\gamma\notin Z_\varphi.$$
\end{theo}
\begin{proof}
	With the consideration above, we prove the statement for $\{e_n\}_{n\in \Z}$ in the Hilbert space $L^2([0,1),p_\varphi)$. 
	First of all we determine the analysis operator $C$ of $\{e_n\}_{n\in \Z}$ in $L^2([0,1),p_\varphi)$. For any $f\in L^2([0,1),p_\varphi)$ we have that $fp_\varphi\in L^1[0,1)$; indeed, by Cauchy-Schwarz inequality and by the fact that $p_\varphi \in L^1[0,1)$ \cite[Sect. 9.2]{Chris},
	$$
	\int_0^1 |f(\gamma)p_\varphi(\gamma)| d\gamma=\int_0^1 |f(\gamma)|p_\varphi(\gamma)^\mez p_\varphi(\gamma)^\mez d\gamma\leq \left (\int_0^1 |f(\gamma)|^2p_\varphi(\gamma)d\gamma \right)^\mez \left (\int_0^1 p_\varphi(\gamma)d\gamma \right)^\mez.
	$$
	Thus, $f\in \D(C)$ if and only if 
	$\sum_{n\in \Z} |\pin{f}{e_n}_{p_\varphi}|^2=\sum_{n\in \Z} |\pin{fp_\varphi}{e_n}|^2<\infty$, i.e. if and only if $fp_\varphi \in L^2[0,1)$. 
	Moreover, for $f\in \D(C)$ we have $Cf=\{\pin{fp_\varphi}{e_n}\}$ and 
	\begin{equation}
		\label{eq_4}
		\pin{Cf}{Cf}_{2}=\int_0^1 |f(\gamma)|^2p_\varphi(\gamma)^2d\gamma.
	\end{equation}
	Hence, $\{e_n\}_{n\in \Z}$ is a lower semi-frame of $L^2([0,1),p_\varphi)$ if and only if there exists $A>0$ such that 
	$$
	A\pin{f}{f}_{p_\varphi}=A\int_0^1 |f(\gamma)|^2p_\varphi(\gamma)d\gamma\leq \int_0^1 |f(\gamma)|^2p_\varphi(\gamma)^2d\gamma, \qquad \forall f\in \D(C),
	$$
	or equivalently,
	\begin{equation}
		\label{dis_3}
		A\int_{Z_\varphi} |f(\gamma)|^2p_\varphi(\gamma)d\gamma\leq \int_{Z_\varphi} |f(\gamma)|^2p_\varphi(\gamma)^2d\gamma, \qquad \forall f\in \D(C).
	\end{equation}
	Now, as consequence of Problem III.2.10 and Example III.2.11 of \cite{Kato}, the inequality \eqref{dis_3} is satisfied if and only if $p_\varphi$ is bounded below by $A$ a.e. in $Z_\varphi$.\\
	For the second part of the statement, let $T$ be the generalized frame operator of $\{e_n\}_{n\in \Z}$. From \eqref{eq_4}, 
	we have, for every $f\in \D(C)$,
	$$
	\pin{Tf}{f}_{p_\varphi}=\pin{Cf}{Cf}_{2}=\int_0^1 |f(\gamma)|^2p_\varphi(\gamma)^2d\gamma=\int_0^1 f(\gamma)p_\varphi(\gamma)\overline{f(\gamma)}p_\varphi(\gamma)d\gamma=\pin{fp_\varphi}{f}_{p_\varphi}.
	$$
	Therefore, $T$ acts as $Tf=fp_\varphi$ on the domain $\D(T)=\{f\in L^2([0,1),p_\varphi): fp_\varphi\in L^2([0,1),p_\varphi)\}$. Consequently, the canonical dual of a lower semi-frame $\{e_n\}_{n\in \Z}$ is $\{\chi_{Z_\varphi}/p_\varphi e_n\}_{n\in \Z}$. Making use of the operator $U$, we finally find the expression of the canonical dual of a lower semi-frame  $\mathcal{T}(\varphi,a)$.
\end{proof}

This result is then the counterpart of the characterization of Bessel sequences of translates in terms of $p_\varphi$. 
Under the assumptions of Theorem \ref{pro_lower_transl} we have the unconditionally convergent reconstruction formula \eqref{weak_rec_1seq(b)} (with $\N$ replaced by $\Z$) 
$$
f=\sum_{n\in \Z} \pin{f}{\varphi_n}\psi_n, \quad \text{ for all  } f\in \langle\varphi\rangle \text{ such that } \left [\wh f,{\wh \varphi}\right]\in L^2(0,1)
$$
where $\psi=\wh \varphi \chi_{Z_\varphi}/p_\varphi$, $\varphi_n(x):=\varphi(x-na)$ and $\psi_n(x):=\psi(x-na)$. 

For a frame  $\mathcal{T}(\varphi,a)$ the expression of the canonical dual is already known \cite[Proposition 4.7, Theorem 4.8]{Ben_Li}. 

The Fourier transform changes translations into modulations. Thus we can also describe the generalized frame operator of the {\it weighted exponentials sequence} $\mathcal{E}(g,b):=\{g_n\}_{n\in \Z}$ of a function $g\in L^2(0,1)$ and parameter $b>0$, where $g_n(x)=e^{2\pi i nb x}g(x)$ for $0<x<1$ (we will often refer to detailed studies in \cite{heil}). Indeed, we can think $g$ as an element of $L^2(\R)$ where $g$ is zero outside $(0,1)$ and transform $\mathcal{E}(g,b)$ into $\mathcal{T}(\wch{g},b)$. Hence the next result can be seen as consequence of Theorem \ref{th_shift}. Again, it is suitable to use the bracket product $[f,h]$ of two elements $f,h\in L^2(0,1)$. We write
\begin{equation}
	\label{eq_bracket2}
	\left [ f,h\right](x):=\frac{1}{b}\sum_{n\in \Z} f\left(\frac{x-n}{b}\right){\ol{h\left (\frac{x-n}{b}\right)}}, \qquad x\in (0,1),
\end{equation}
with the meaning that $f$ and $h$ are zero outside $(0,1)$. So \eqref{eq_bracket2} is actually a {\it finite} sum, but for simplicity of notation we write an infinite sum in \eqref{eq_bracket2} and in the theorem below.

\begin{cor}
	\label{th_weigh_exp} 
	Let $g\in L^2(0,1)$ and $b>0$. Let $C_g$ and $T_g$ be the analysis and generalized frame operators of $\mathcal{E}(g,b):=\{g_n\}_{n\in \Z}$, respectively. The following statements hold.
	\begin{enumerate}
		\item The operator $C_g$ has domain $\D(C_g)=\{f\in L^2(0,1):[f,{g}]\in L^2(0,1)\}$ and $C_g f=\{\pin{[f,{g}] }{e_n}\}_{n\in \Z}$ for all $f\in \D(C_g)$. The domain $\D(C_g)$ contains all bounded functions with compact support in $(0,1)$, thus $\D(C_g)$ is dense.
		\item The domain $\D(T_g)$ of $T_g$ is the subspace of  $f\in L^2(0,1)$ such that
		\begin{enumerate}
			\item[\emph{(a)}] $\left [ f,{ g}\right]\in L^2(0,1)$;
			\item[\emph{(b)}] 
			the function from $(0,1)$ to $\C$ given by $x\mapsto{g} (x)\sum_{n\in \Z} f\left(x-\frac{n}{b}\right)\ol{{g}\left (x-\frac{n}{b}\right)}$ belongs to $L^2(0,1)$,
		\end{enumerate}
		and 
		\begin{equation}
			\label{eq_T_g_1}
			{T_g f}(x)=\frac{1}{b}{g} (x)\sum_{n\in \Z} f\left(x-\frac{n}{b}\right)\ol{{g}\left (x-\frac{n}{b}\right)},
		\end{equation}
		for a.e. $x \in (0,1)$ and $f\in \D(T_g)$.
	\end{enumerate}
\end{cor}

As a particular case, for $0<b\leq1$ we get the following.

\begin{cor}
	\label{cor_exp_b<=1}
	Let $g\in L^2(0,1)$ and $0<b\leq 1$. Let $C_g$ and $T_g$ be the analysis and generalized frame operators of $\mathcal{E}(g,b):=\{g_n\}_{n\in \Z}$, respectively. The following statements hold.
	\begin{enumerate}
		\item The operator $C_g$ has domain $\D(C_g)=\{f\in L^2(0,1):fg\in L^2(0,1)\}$ and $C_gf=\{\frac{1}{b}\pin{f\ol{g}}{e_n}\}$ for $f\in \D(C_g)$. 
		\item The operator $T_g$ is the multiplication operator by $\frac{1}{b}|g|^2$, i.e.
		$\D(T_g)=\{f\in L^2(0,1):f|g|^2\in L^2(0,1)\}$ and $T_gf=\frac{1}{b}f|g|^2$ for $f\in \D(T_g)$. 
	\end{enumerate}
\end{cor}

We recall some facts about $\mathcal{E}(g,1)$ that can be found in \cite[Theorem 10.10]{heil}. 
\begin{rem}
	\label{rem_6.4}
	The sequence $\mathcal{E}(g,1)$ is 
	\begin{itemize}
		\item complete in $L^2(0,1)$ if and only if $g(x)\neq 0$ a.e. in $(0,1)$;
		\item minimal if and only if $g(x)\neq 0$ a.e.\ and $1/\ol{g}\in L^2(0,1)$ (and in this case its biorthogonal sequence is $\mathcal{E}(1/\ol{g},1)$);
		\item a Bessel sequence if and only if $g$ is bounded above a.e. in $(0,1)$; 
		\item a Riesz basis of $L^2(0,1)$ if and only if it is a frame of $L^2(0,1)$  if and only if $g$ is bounded above and away from zero a.e. in $(0,1)$.
	\end{itemize}
\end{rem}

By Theorem \ref{pro_lower_transl} we can state the following. 
\begin{cor}
	\label{cor_6.5}
	Let $0<b\leq1$. The sequence $\mathcal{E}(g,b)$ is a lower semi-frame of $L^2(0,1)$ if and only if $g$ is bounded away from zero a.e. in $(0,1)$. Moreover, if $\mathcal{E}(g,b)$ is a lower semi-frame, then the canonical dual is the Bessel sequence $\mathcal{E}(b/\ol{g},b)$.
\end{cor}

\begin{rem}
	\label{rem_upp_trans}
	The previous result is in line with Proposition \ref{car_seq_T_xi} and Corollary \ref{cor_exp_b<=1}.  \\	
	As a continuation of a discussion in Section \ref{secSeq_1seq}, when $\mathcal{S}(g,b)$ for $0<b\leq1$ is an upper semi-frame, then we cannot always find a dual lower semi-frame by inverting the frame operator. Indeed, if  $\mathcal{S}(g,b)$  is an upper semi-frame for $0<b\leq1$, then  $\mathcal{E}(g,1)\subset R(S_g)=R(T_g)$ if and only if $1/g\in L^2(0,1)$, which is not always the case.
\end{rem}

Throughout the rest of the section we are going to often use the characterization below. It is stated in \cite{nielsikic} and it is a direct consequence of Theorem 8 of the fundamental paper of Hunt, Muckenhoupt and Wheeden \cite{huntmuckwhee}. 
We say that $\omega\in L^1(0,1)$ belongs to the
{\it Muckenhoupt's class} $\mathcal{A}_2(0,1)$ if $\omega(x)> 0$ for a.e. $x\in (0,1)$ and
$$
\sup_I \left(\frac{1}{|I|}\int_I \omega(x) dx\right)\left( \frac{1}{|I|}\int_I \frac{1}{\omega(x)}dx\right )<\infty
$$ 
where the supremum is taken over all intervals $I\subseteq (0,1)$. \\
The characterization is the following (see also \cite[Theorem 10.19]{heil}): a sequence of translates $\mathcal{T}(\varphi,a)$ of $\varphi\in L^2(\R)$ is a Schauder basis with respect to the ordering 
\begin{equation}
	\label{ord_Z}
	\Z=\{0,1,-1,2,-2,\dots\}
\end{equation}
if and only if $p_\varphi\in A_2(0,1)$. 

In the setting of weighted exponentials this result is formulated as follows (see \cite[Theorem 5.15]{heil}). Let $g\in L^2(0,1)$. Then the sequence $\mathcal{E}(g,1)=\{g_n\}_{n\in \Z}$ is a Schauder basis with respect to the ordering $\Z=\{0,1,-1,2,-2,\dots\}$ if and only if $|g|^2\in A_2(0,1)$.

\begin{exm}
	Every nonnegative measurable function $\omega:(0,1)\to \C$, which is bounded above and away from zero a.e., belongs to $A_2(0,1)$. The converse is not true. For instance, the function $\omega(x)=x^\alpha$ for $x\in (0,1)$ and $|\alpha|< 1$ is in $A_2(0,1)$ (see \cite[Example 7.1.7]{Grafakos}).
\end{exm}

As promised in Section \ref{subsecSeq_low_rec} we give a new example of a lower semi-frame that does not generate duality on the whole space. It is formulated as weighted exponentials sequence and takes advantage of the characterization above for a suitable function $g$.

\begin{exm}
	\label{countex_transl}
	Suppose that a weighted exponentials sequences $\mathcal{E}(g,b)=\{g_n\}_{n\in \Z}$, ordered as in \eqref{ord_Z}, is a lower semi-frame of $L^2(0,1)$. This means that $g$ is bounded away from zero, in particular $\{g_n\}_{n\in \Z}$ is minimal by Corollary \ref{cor_6.5} and Remark \ref{rem_6.4}.  
	If the duality in \eqref{weak_rec_1seq(b)} extends to the whole $L^2(0,1)$, i.e.
	$$
	f=\sum_{n\in \Z} \pin{f}{g_n}e_n/\ol{g} \qquad \forall f\in L^2(0,1),
	$$
	then $\{g_n\}_{n\in \Z}$ turns out to be a Schauder basis. Hence, as said after \eqref{ord_Z}, 
	$|g|^2\in A_2(0,1)$. What we are going to do is to construct a function $g\in L^2(0,1)$ which is bounded away from zero and such that $|g|^2\notin A_2(0,1)$. \\
	We begin by noting that the series $\sum_{j=2}^\infty j^{-3j}$ is convergent to a positive number less than $1$. Thus we can define for $x\in (0,1)$	$$g(x)=\sum_{k=2}^{\infty} k^k \chi_{(\sum_{j=2}^{k-1}j^{-3j},\sum_{j=2}^{k}j^{-3j})}(x)+\chi_{(\sum_{j=2}^\infty j^{-3j},1)}(x)$$
	where $\sum_{j=2}^{1}j^{-3j}$ is $0$ by definition and $\chi_{(a,b)}$ denotes the characteristic function of $(a,b)$. Also $1/g$ is a well-defined function on $[0,1]$,
	$$1/g(x)=\sum_{k=2}^{\infty} k^{-k} \chi_{(\sum_{j=2}^{k-1}j^{-3j},\sum_{j=2}^{k}j^{-3j})}(x)+\chi_{(\sum_{j=2}^\infty j^{-3j},1)}(x),$$
	and $g$ is bounded away from zero (with bound $1$). We have that $g,1/g\in L^2(0,1)$. Indeed  
	$$
	\int_0^1 g(x)^2dx=\sum_{k=2}^{\infty} k^{2k}k^{-3k}+1-\sum_{k=2}^{\infty} k^{-3k}<2\sum_{k=2}^{\infty} k^{-2}+1<\infty,
	$$
	$$
	\int_0^1 \frac{1}{g(x)^{2}}dx=\sum_{k=2}^{\infty} k^{-2k}k^{-3k}+1-\sum_{k=2}^{\infty} k^{-3k}<2\sum_{k=2}^{\infty} k^{-6}+1<\infty.
	$$
	Now we prove that $g^2\notin A_2(0,1)$. Let $k>2$ and $I_k=[\sum_{j=2}^{k}j^{-3j}-\epsilon_k,\sum_{j=2}^{k}j^{-3j}+\epsilon_k]$ where $0<\epsilon_k<(k+1)^{-3(k+1)}$. The interval $I_k$ is chosen so that the first half part of $I_k$ is contained in $[\sum_{j=2}^{k-1}j^{-3j},\sum_{j=2}^{k}j^{-3j}]$, where the value of $g$ is $k^k$, and the second half part of $I_k$ is contained in $[\sum_{j=2}^{k}j^{-3j},\sum_{j=2}^{k+1}j^{-3j}]$, where the value of $g$ is $(k+1)^{(k+1)}$. Therefore 
	$\frac{1}{|{I_k}|}\int_{I_k} g(x)^2dx=\frac{1}{2}(k^{2k}+(k+1)^{2(k+1)})$. In a similar way, $\frac{1}{|{I_k}|}\int_{I_k} \frac{1}{g(x)^2}dx=\frac{1}{2}(k^{-2k}+(k+1)^{-2(k+1)})$; so 
	\begin{align*}
		\frac{1}{|I_k|}\int_{I_k} g(x)^2dx\frac{1}{|{I_k}|}\int_{I_k} \frac{1}{g(x)^2}dx&=\frac{1}{4}(2+k^{2k}(k+1)^{-2(k+1)}+(k+1)^{2(k+1)}k^{-2k})\\
		&\geq \frac{1}{4}(k+1)^2.
	\end{align*}
	Therefore 
	$$\sup_{k>2}\left (\frac{1}{|I_k|}\int_{I_k} g(x)dx\frac{1}{|{I_k}|}\int_{I_k} \frac{1}{g(x)}dx\right)=\infty,$$ i.e. $g^2\notin A_2(0,1)$. The desired example is now concluded. 
\end{exm}

We now make some comparisons with Remark \ref{rem_cntrex_lower} to highlight the features of this example. 
Both Example \ref{countex_transl} and Remark \ref{rem_cntrex_lower}(i) consist of a  lower semi-frame $\xi$ with $\D(C_\xi)$ dense and the norm of the elements is constant. 
However, the sequence in Remark \ref{rem_cntrex_lower}(i) is quite abstract, indeed it involves an infinite Hilbert spaces sum.  Example \ref{countex_transl} has instead the quality of being a very simple sequence (indeed of weighted exponentials).  
On the other hand, the lower semi-frame $\xi=\{e_1+e_n\}_{n\geq 2}$ in Remark \ref{rem_cntrex_lower}(ii) is also very simple, but $\D(C_\xi)$ is not dense, in contrast to Example \ref{countex_transl}.

We end with two more examples showing that for a lower semi-frame $\xi$ the frame operator $S_\xi$ may be not closed and may be conditionally convergent. 
In contrast, we have seen that the generalized frame operator $T_\xi$ has better properties: it is self-adjoint and unconditionally defined. Once again the examples involve sequences of weighted exponentials.

\begin{exm}
	\label{exm_S_notclosed}
	Let $g\in L^2(0,1)$ be bounded away from zero but not above, $g>0$ a.e. and $g^2\notin A_2(0,1)$ (for instance the function $g$ in Example \ref{countex_transl}). We consider the sequence $\mathcal{E}(g,1)=\{g_n\}_{n\in \Z}$ ordered as in \eqref{ord_Z} and suppose that the frame operator $S_g$ of $\mathcal{E}(g,1)$ is closed. By Corollary \ref{cor_exp_b<=1}, $T_g$ is the multiplication operator $M_{g^2}$ by $g^2$ and we have also $\mathcal{E}(1/g,1)\subset \D(S_g)$. Denoting by $\ol{L}$ the closure of the restriction $L$ of $M_{g^2}$ to span$(\mathcal{E}(g,1))$, we can then affirm  that $\ol{L}\subseteq S_g$. \\
	Actually $\ol{L}$ coincide with $M_{g^2}$. Indeed, the closure of span$(\mathcal{E}(g,1))$, under the graph norm $\nor_{M_{g^2}}$ of $M_{g^2}$, is $\D(M_{g^2})$. To prove this, let $h\in \D(M_{g^2})$ such that $0=\pin{1/ge_n}{h}_{M_{g^2}}=\pin{1/ge_n}{h}+\pin{e_ng}{g^2h}=\pin{e_n}{h(g^3+1/g)}$ for all $n\in \Z$.  Thus the function $h(g^3+1/g)$ of $L^1(0,1)$ is null a.e. in $(0,1)$, i.e. $h=0$ a.e. in $(0,1)$. \\
	In conclusion, $S_g=M_{g^2}=T_g$, which contradicts Proposition \ref{pro_cs_S=T_2}. Therefore, $S_g$ is not closed. 
\end{exm}

\begin{exm}
	\label{ST_translates}
	Let $g\in L^2(0,1)$ and $\xi:=\mathcal{E}(g,1):=\{g_n\}_{n\in \Z}$ ordered as in \eqref{ord_Z}. 
	Since $\n{g_n}$ is constant for all $n\in \Z$, by \cite[Lemma 3.6.9]{Chris}, $\xi$ is an unconditional Schauder basis of $L^2(0,1)$ if and only if $\xi$ is a Riesz basis of $L^2(0,1)$, i.e. $g$ is bounded above and away from zero a.e. in $(0,1)$. \\	
	Assume that $|g|^2\in A_2(0,1)$ is bounded away from zero but not above. Hence $\xi:=\{g_n\}_{n\in \Z}$ is a lower semi-frame and Schauder basis with respect to the ordering $\Z=\{0,1,-1,2,-2,\dots\}$, but it is not an unconditional Schauder basis. Then there exists a bijection $\sigma:\Z\to \N$ such that $\xi'=\{g_{\sigma(n)} \}_{n\in \Z}$ is not a Schauder basis. By Proposition \ref{pro_cs_S=T_2} and recalling that $T_\xi$ is unconditionally defined, we have $S_{\xi'}\neq T_{\xi'}=T_{\xi}=S_{\xi}$. In conclusion, the frame operator of $\mathcal{E}(g,1)$ depends on the chosen ordering. 
\end{exm}

\section*{Conclusions}

In this paper we have seen how the operator $T_\xi$ is more appropriate than $S_\xi$ in the context of non-Bessel sequences.  
To summarize, $T_\xi$ is unconditionally defined, positive self-adjoint (in the space $\H_\xi=\ol{\D(C_\xi)}[\nor]$) and when $\xi$ is a lower semi-frame the canonical Bessel dual of $\xi$ is defined in terms of $T_\xi^{-1}$, as $T_\xi^{-1}P\xi$ (where $P$ is the projection onto $\H_\xi$). Moreover, $T_\xi^{-1/2}P\xi$ is a Parseval frame of $\H_\xi$.

\section*{Acknowledgments}

This work was partially supported by the ``Gruppo Nazionale per l'Analisi Matematica, la Probabilit\`{a} e le loro Applicazioni'' (GNAMPA-INdAM).

\vspace*{0.5cm}
\begin{center}
\textsc{Rosario Corso, Dipartimento di Matematica e Informatica} \\
\textsc{Università degli Studi di Palermo, I-90123 Palermo, Italy} \\
{\it E-mail address}: {\bf rosario.corso02@unipa.it}
\end{center}

\end{document}